\documentclass[pdftex,11pt]{article}
\usepackage[utf8]{inputenc}
\usepackage[english]{babel}
\usepackage{amsmath,amssymb,amsthm}
\usepackage{url}
\usepackage{graphicx}
\usepackage{nicefrac}

\usepackage{caption}
\usepackage{subcaption}

\usepackage[pdftex,
    colorlinks=true,
    urlcolor=blue,       
    filecolor=green,     
    linkcolor=red,       
    pdfauthor={Max A. Alekseyev and Toby Berger},
    pdfpagemode=None,
    bookmarksopen=true]{hyperref}

\usepackage{vmargin}
\setpapersize{USletter}
\setmargnohfrb{1.0in}{1.0in}{1.0in}{1.0in}
\newtheorem{Th}{Theorem}
\newtheorem{Lemma}[Th]{Lemma}

\title{Solving the Tower of Hanoi with Random Moves} 

\author{
Max A. Alekseyev\thanks{Corresponding author. Email: {\tt maxal@gwu.edu}}~\thanks{Department of Mathematics, George Washington University.}
\quad and\quad 
Toby Berger\thanks{Department of Electrical and Computer Engineering, University of Virginia.} 
}

\date{}

\begin{document} 
\maketitle\thispagestyle{empty}

\section{Introduction}

The Tower of Hanoi puzzle consists of $n$ disks of distinct sizes distributed across $3$ pegs.
We will refer to a particular distribution of the disks across the pegs as a \emph{state} and call it \emph{valid}
if on each peg disks form a pile with the disk sizes decreasing from bottom up.
Since each disk can reside at one of $3$ pegs, while the order of disks on each peg is uniquely defined by their sizes, the total number of valid states is $3^n$.
At a single move it is allowed to transfer a disk from the top of one peg to the top of another peg if this results in a valid state.
In the classic formulation of the Tower of Hanoi puzzle, all disk initially are located on the first peg and it is required to transfer them all 
to the third peg with the smallest number of moves, which is known to be $2^n-1$.

French mathematician Edouard Lucas invented the Tower of Hanoi puzzle in 1883~\cite{Lucas1893}.  
Apparently, he simultaneously created the following legend~\cite{Ripley}:
\begin{quote}
``Buddhist monks somewhere in Asia are moving 64 heavy gold rings from peg 1 to peg 3.  
When they finish, the world will come to an end!"
\end{quote}
That Hanoi is located in what was then French Indo-China perhaps explains why a Frenchman saw fit to include Hanoi in the name of his puzzle.  
However, the legend never placed the monks and their tower explicitly in Hanoi or its immediate environs.  
Lucas' Tower of Hanoi puzzle became an international sensation (think Loyd's 15 puzzle and Rubik's cube); the legend was used to bolster sales.  
Still a popular and beloved toy, the Tower of Hanoi now also can be accessed over the Internet as a computer applet.

In the current work, we study solution of the Tower of Hanoi puzzle and some of its variants with random moves, where each move is chosen uniformly from 
the set of the valid moves in the current state. We prove the exact formulae for the expected number of random moves to solve the puzzles
and further present an alternative proof for one of the formulae that couples a theorem about expected commute times of random walks
on graphs with the delta-to-wye transformation used in the analysis of three-phase AC systems for electrical power distribution.

\section{Puzzle Variations and Preliminary Results}

If the valid states of the Tower of Hanoi represented as nodes of a graph and every two states that are one move away from each other are connected with an edge, 
the resulting graph is known as Sierpinski gasket (Fig.~\ref{fig:Sierp123}). In other words, Sierpinski gasket represents the \emph{state transition diagram} of the the Tower of Hanoi.
Namely, the Tower of Hanoi with $n=1$ disk corresponds to a graph with three nodes $A$, $B$, and $C$ and 
three edges $\overline {\rm AB}, \overline {\rm BC}$, and $\overline {\rm AC}$ (Fig.~\ref{fig:Sierp123}a).
The nodes $A$, $B$ and $C$ correspond, respectively, to the three possible states: ``$D_1$ is on the first peg", ``$D_1$ is on the second peg", 
and ``$D_1$ is on the third peg". The presence of edge $\overline{\rm AC}$ represents the fact that it is possible to reach $C$ from $A$ or $A$ from $C$ in one move; 
the other two edges have analogous interpretations.

\begin{figure}[!t]
\begin{tabular}{l}
\begin{tabular}{ll}
a) & b)\\
\begin{tabular}{c}
\includegraphics[width=.3\textwidth]{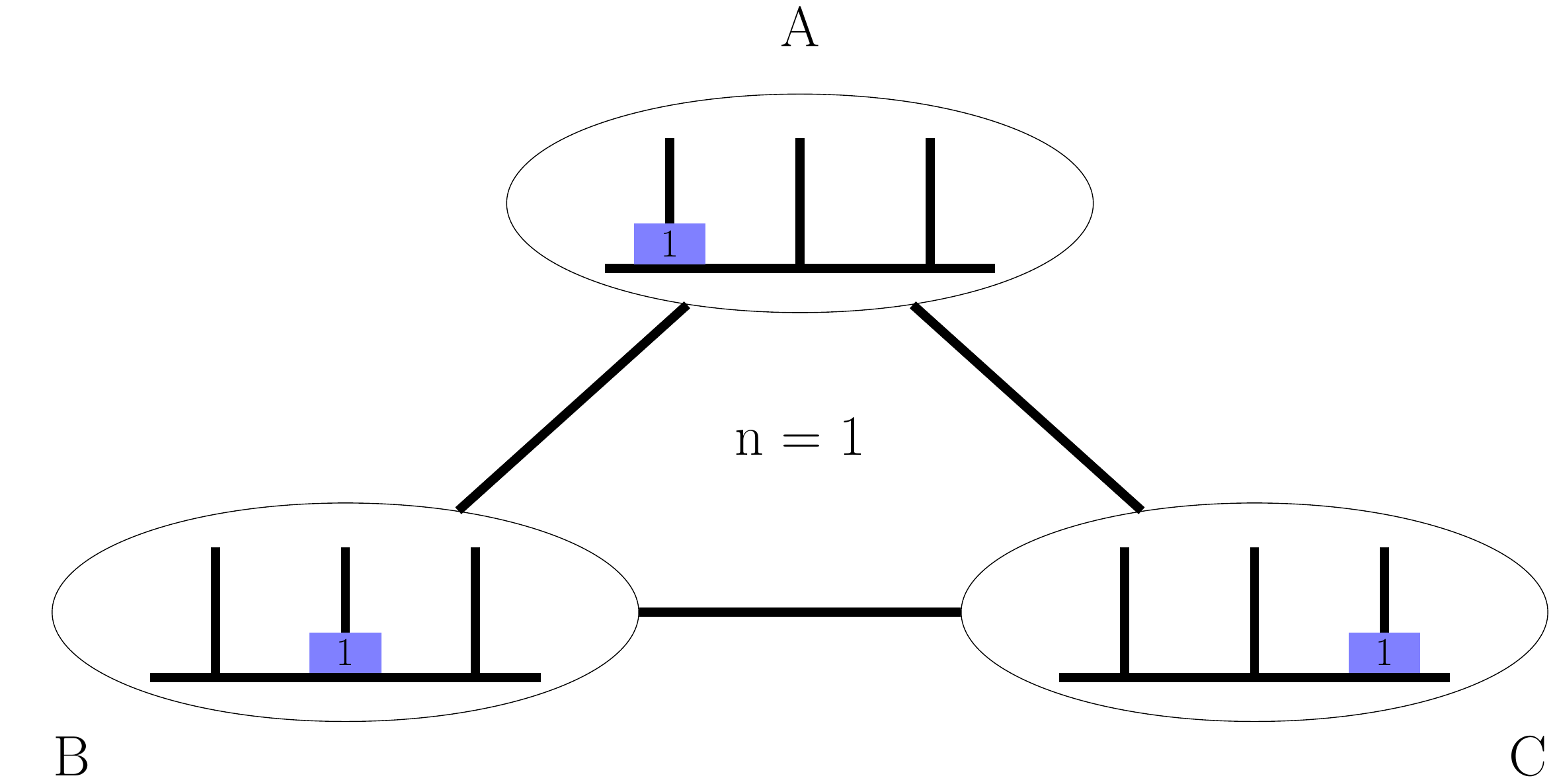} 
\end{tabular}
&
\begin{tabular}{c}
\includegraphics[width=.65\textwidth]{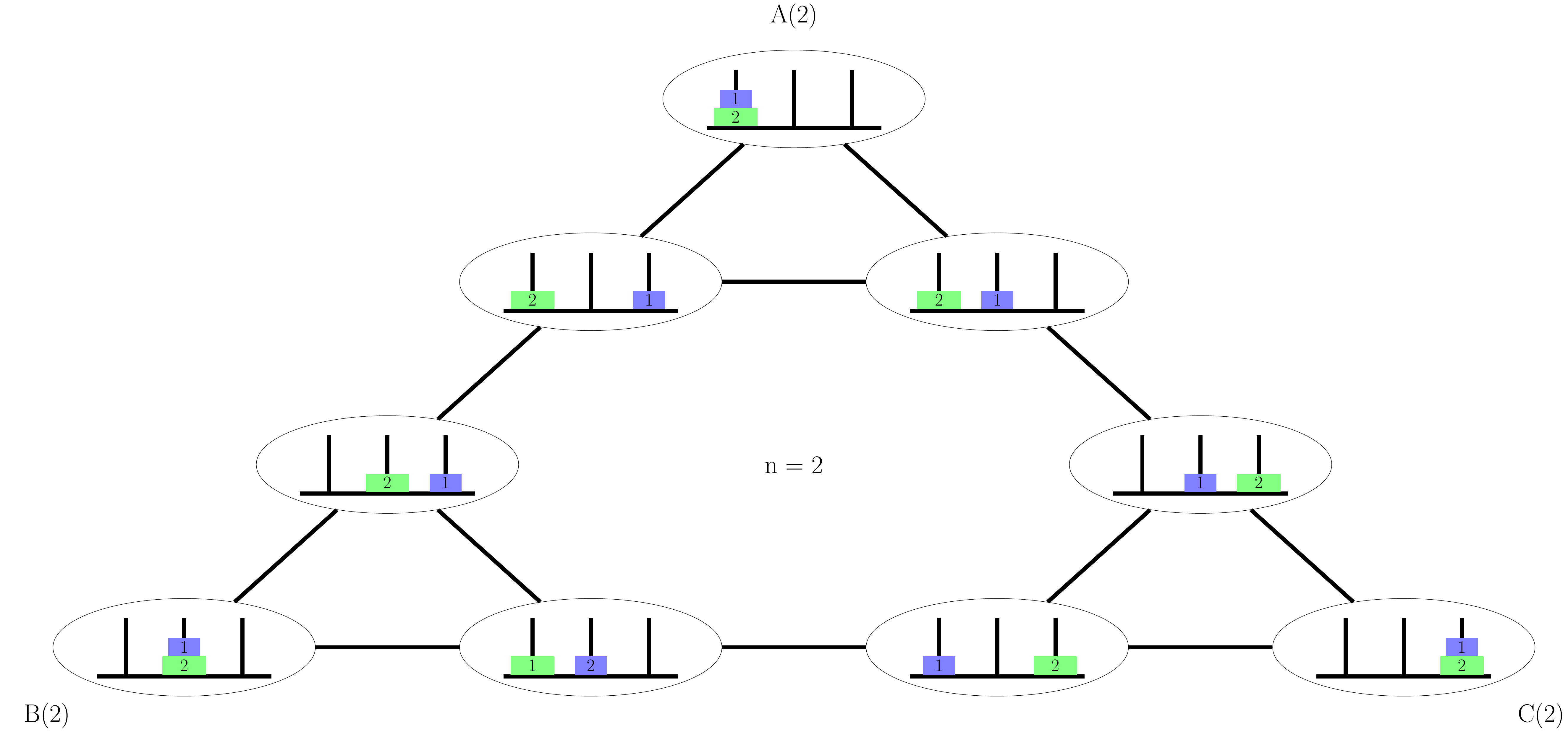} 
\end{tabular}
\end{tabular}
\\
c)\\
\begin{tabular}{c}
\includegraphics[width=0.95\textwidth]{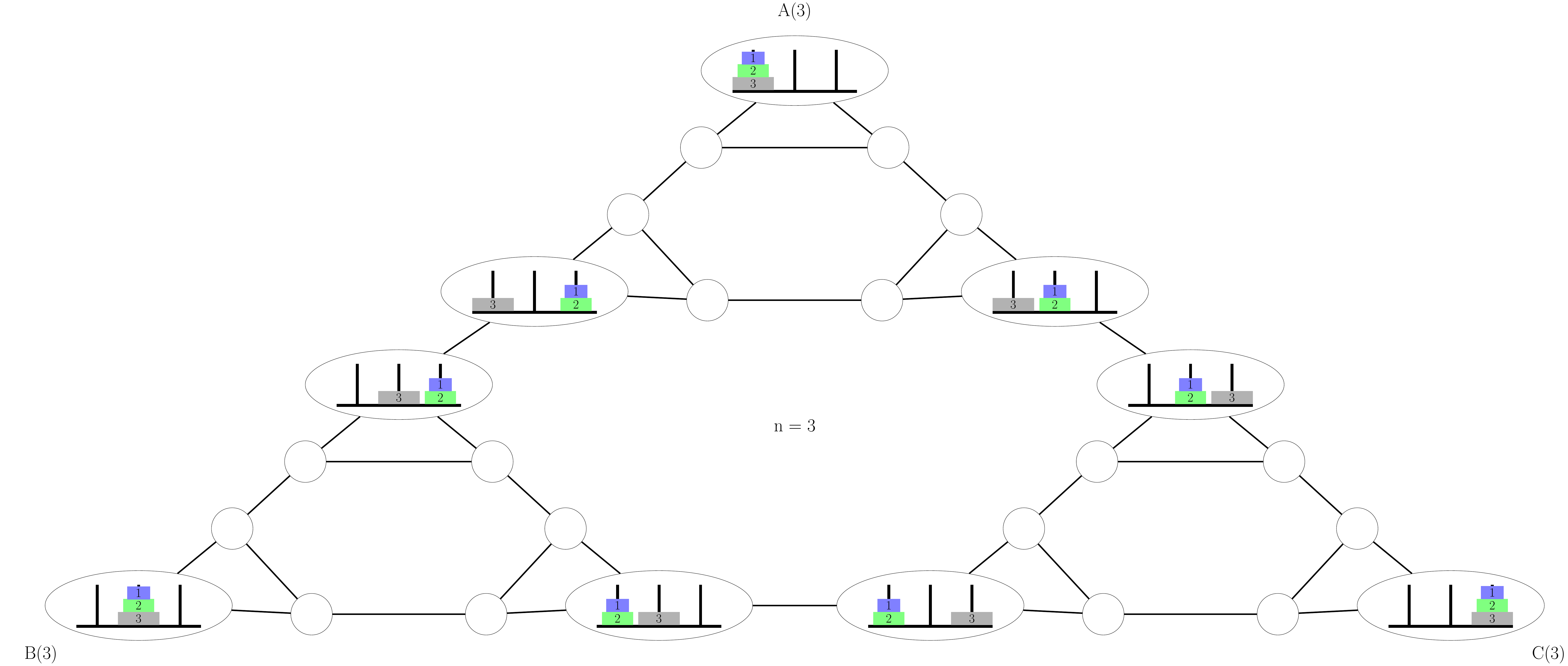} 
\end{tabular}
\end{tabular}
\caption{Sierpinski gaskets corresponding to \textbf{a)} the state transition diagram for the 1-disk Tower of Hanoi;
\textbf{b)} the state transition diagram for the 2-disk Tower of Hanoi, 
which is composed of three replicas of the 1-disk diagram with an added disk 2; 
and \textbf{c)} the state transition diagram for the 3-disk Tower of Hanoi, 
which is composed of three replicas of the 2-disk diagram with an added disk 3 (only corner states in each replica are labeled).}
\label{fig:Sierp123}
\end{figure}

The graph for the state transitions of the Tower of Hanoi with $n=2$ disks is obtained by arranging three replicas of the graph for $n=1$ 
with an added disk $D_2$ at a fixed peg in each replica, 
and then connecting each of the resulting three pairs of nearest neighbor nodes by bridging links (Fig.~\ref{fig:Sierp123}b).
This leaves only three corner nodes which we label $A(2)$, $B(2)$, and $C(2)$ such that
$A(2)$ corresponds to both disks being on the first peg, $B(2)$ corresponds to both disks being on the second peg,
and $C(2)$ corresponds to both disks being on the third peg. 

Similarly, the graph for the state transitions of the Tower of Hanoi with $n=3$ disks is obtained by arranging three replicas
of that for $n=2$ in the same way that was done to get the $n=2$ graph from the $n=1$ graph (Fig.~\ref{fig:Sierp123}c). 
In general, for any positive integer $k$, the state transition diagram for $n=k+1$ is obtained from that for $n=k$ 
by another iteration of this procedure employing three replicas and three bridges. 

The classic Tower of Hanoi puzzle corresponds to finding the shortest path between two corner nodes in the corresponding Sierpinski gasket.
We consider the following variants of the Tower of Hanoi puzzle with $n$ disks solved with random moves (which correspond to random walks in the Sierpinski gasket):

\begin{description}
\item[$r\to a$:]
The starting state is \emph{random} (chosen uniformly from the set of all $3^n$ states).
The final state is with all disks on the same (\emph{any}) peg.

\item[$1\to 3$:]
The starting state is with all disks on the \emph{first} peg. The final state is with all disks on the \emph{third} peg.

\item[$1\to a$:]
The starting state is with all disks on the \emph{first} peg. The final state is with all disks on the same (\emph{any}) peg.
At least one move is required.

\item[$\nicefrac12\to a$:]
The starting state is with the largest disk on the \emph{second} peg and the other disks on the \emph{first} peg.
The final state is with all disks on the same (\emph{any}) peg.

\item[$r\to 1$:]
The starting state is \emph{random} (chosen uniformly from the set of all $3^n$ states).
The final state is with all disks on the first peg.
\end{description}

Let $E_X(n)$ the expected number of random moves required to solve Puzzle $X$ with $n$ disks. 
The puzzles described above are representative for classes of similar puzzles obtained by renaming pegs. 
In particular, we can easily get the following identities:
$$E_{1\to 3}(n) = E_{1\to 2}(n) = E_{2\to 3}(n) = E_{3\to 1}(n) = E_{2\to 1}(n) = E_{3\to 2}(n),$$
$$E_{1\to a}(n) = E_{2\to a}(n) = E_{3\to a}(n),$$
$$E_{\nicefrac12\to a}(n) = E_{\nicefrac13\to a}(n) = E_{\nicefrac23\to a}(n) = E_{\nicefrac21\to a}(n) = E_{\nicefrac31\to a}(n) = E_{\nicefrac32\to a}(n),$$
$$E_{r\to 1}(n) = E_{r\to 2}(n) = E_{r\to 3}(n).$$

Puzzle~$r\to a$ was posed by David G. Poole who submitted values $E_{r\to a}(n)$ for $n$ up to $5$ as sequence \texttt{A007798} to 
the Online Encyclopedia of Integer Sequences (OEIS)~\cite{OEIS}.
Later Henry Bottomley conjectured the following formula for $E_{r\to a}(n)$:
\begin{equation}\label{FN1}
E_{r\to a}(n) = \frac{5^n-2\cdot 3^n + 1}{4}.
\end{equation}
Puzzle~$1\to 3$ was posed by the second author~\cite{Berger08} 
who also submitted numerators of $E_{1\to 3}(n)$ for $n$ up to 4 as sequence \texttt{A134939} to the OEIS
but did not conjecture a general formula for $E_{1\to 3}(n)$.

Below we will prove formula \eqref{FN1} as well as the 
following formula for $E_{1\to 3}(n)$:
\begin{equation}\label{FN2}
E_{1\to 3}(n) = \frac{(3^n - 1)(5^n - 3^n)}{2\cdot 3^{n-1}},
\end{equation}
which was originally announced by the first author in 2008.
We will also prove the following formulae for other puzzles:
\begin{equation}\label{FN3}
E_{1\to a}(n) = \frac{3^n - 1}{2},
\end{equation}
\begin{equation}\label{FN4}
E_{\nicefrac12\to a}(n) = \frac{3}{2}(5^{n-1} - 3^{n-1}),
\end{equation}
and
\begin{equation}\label{FN5}
E_{r\to 1}(n) = \frac{5^{n+1}-2\cdot 3^{n+1} + 5}{4} - \left(\frac{5}{3}\right)^n = \frac{(3^n-1)(5^{n+1}-2\cdot 3^{n+1})+5^n-3^n}{4\cdot 3^n}.
\end{equation}
We summarize these formulae along with references to the OEIS in Table~\ref{tab:EX}.

\begin{table}
\begin{center}
\begin{footnotesize}
\begin{tabular}{l||c|l|l}
\hline
Puzzle $X$ & Formula for $E_X(n)$ & Initial values ($n=1,2,\dots$) & Sequences in the OEIS\\
\hline\hline
$r\to a$ & $\frac{5^n-2\cdot 3^n + 1}{4}$ & $0, 2, 18, 116, 660, \ldots$ 
& $\texttt{A007798}(n)$ \\
\hline
$1\to 3$ & $\frac{(3^n - 1)(5^n - 3^n)}{2}\,/\,3^{n-1}$ & $2, \nicefrac{64}{3}, \nicefrac{1274}{9}, \nicefrac{21760}{27}, \ldots$
& $\texttt{A134939}(n)\,/\,\texttt{A000244}(n-1)$ \\
\hline
$1\to a$ & $\frac{3^n - 1}{2}$ & $1, 4, 13, 40, 121, 364, \ldots$
& $\texttt{A003462}(n)$ \\
\hline
$\nicefrac12\to a$  & $\frac{3\cdot (5^{n-1} - 3^{n-1})}{2}$ & $0, 3, 24, 147, 816, 4323, \ldots$
& $\texttt{A226511}(n-1)$ \\
\hline
$r\to 1$  & $\frac{(3^n-1)(5^{n+1}-2\cdot 3^{n+1})+5^n-3^n}{4}\,/\,3^n$ & $\nicefrac{4}{3}, \nicefrac{146}{9}, \nicefrac{3034}{27}, \nicefrac{52916}{81}, \ldots$
& $\texttt{A246961}(n)\,/\,\texttt{A000244}(n)$ \\
\hline
\end{tabular}
\caption{Variations of the Tower of Hanoi puzzle with $n$ disks, the expected numbers of random moves required to solve them, and references to the corresponding sequences in the OEIS~\cite{OEIS}.}
\label{tab:EX}
\end{footnotesize}
\end{center}
\end{table}

\section{Lemmas and Proofs}

Without loss of generality assume that the $n$-disk Tower of Hanoi has disks of sizes $1,2,\dots,n$.
We denote the disk of size $k$ by $D_k$ so that $D_1$ and $D_n$ refer to the smallest and largest disks, respectively.
Similarly, we let $D_k^m$ ($k\geq m$) be the set of all disks of sizes from $m$ to $k$ inclusively.

In the solution of Puzzle $1\to a$ with random moves, let
$p_1(n)$ and $p_2(n)$ denote the probability that a final state 
has all disks on the first and second peg, respectively.
From the symmetry, it is clear that the probability 
that a final state has all disks on the third peg is also $p_2(n)$,
so $p_1(n)+2 p_2(n)=1$.

Similarly, in solution to Puzzle~$\nicefrac12\to a$ with random moves,
let $q_1(n),\ q_2(n),$ and $q_3(n)$ denote the probability that 
a final state has all disks on the first, second, and third peg, respectively.

The relationships of Puzzles $1\to a$ and $\nicefrac12\to a$ to Puzzles $r\to a$ and $1\to 3$
are given by Lemmas~\ref{ThP1} and \ref{ThP2} below.

\begin{Lemma}\label{ThP1}
$$E_{r\to a}(n) = E_{r\to a}(n-1) + \frac{2}{3} E_{\nicefrac12\to a}(n).$$
\end{Lemma}
\begin{proof}
It is easy to see that $D_n$ cannot move unless $D_{n-1}^1$ are on the same peg.
In Puzzle~$r\to a$, the expected number of random moves required to arrive at such state is $E_{r\to a}(n-1)$.
Moreover, since the starting state is uniformly chosen, in the final state $D_{n-1}^1$
will be on any peg with the equal probability $\nicefrac{1}{3}$. In particular, with the probability $\nicefrac{1}{3}$
it is the same peg where $D_n$ resides and the puzzle is solved. Otherwise, with the probability $\nicefrac{2}{3}$ the disks $D_{n-1}^1$ and $D_n$
are on distinct pegs and thus we can view the remaining moves as solving an instance of Puzzle~$\nicefrac12\to a$. Therefore, 
$$E_{r\to a}(n) = \frac{1}{3}E_{r\to a}(n-1) + \frac{2}{3} (E_{r\to a}(n-1)+E_{\nicefrac12\to a}(n)) = E_{r\to a}(n-1) + \frac{2}{3} E_{\nicefrac12\to a}(n).$$
\end{proof}

\begin{Lemma}\label{ThP2}
$$E_{1\to 3}(n) = \frac{E_{1\to a}(n)}{p_2(n)}.$$
\end{Lemma}

\begin{proof} 
In the course of solving Puzzle~$1\to 3$ with random moves, all disks will first appear on the same peg after $E_{1\to a}(n)$ moves on average.
This peg will be the first peg with the probability $p_1(n)$, the second peg with the probability $p_2(n)$, or
the third peg also with the probability $p_2(n)$. In the last case Puzzle~$1\to 3$ is solved,
while in the first two cases we basically obtain a new instance of Puzzle~$1\to 3$.
Therefore, $E_{1\to 3}(n) = E_{1\to a}(n) + (p_1(n)+p_2(n)) E_{1\to 3}(n)$, implying that $E_{1\to 3}(n) = \frac{E_{1\to a}(n)}{p_2(n)}.$
\end{proof}

Lemmas~\ref{ThP1} and \ref{ThP2} imply that explicit formulae for $E_{r\to a}(n)$ and $E_{1\to 3}(n)$ 
easily follow from those for $E_{1\to a}(n)$, $E_{\nicefrac12\to a}(n)$, and of $p_2(n)$.

\begin{Lemma} The following equalities hold:
\begin{itemize}
\item[(i)] $E_{1\to a}(n) = E_{1\to a}(n-1) + 2 p_2(n-1) E_{\nicefrac12\to a}(n)$
\item[(ii)] $p_1(n) = p_1(n-1) + 2 p_2(n-1) q_2(n)$
\item[(iii)] $p_2(n) = p_2(n-1) q_1(n) + p_2(n-1) q_3(n)$
\item[(iv)] $E_{\nicefrac12\to a}(n) = \tfrac{1}{2} + E_{1\to a}(n-1) + (p_1(n-1)+p_2(n-1)) E_{\nicefrac12\to a}(n)$
\item[(v)] $q_1(n) = p_1(n-1) q_1(n) + p_2(n-1) q_3(n)$
\item[(vi)] $q_2(n) = \frac{3}{4}( p_2(n-1) + (p_1(n-1)+p_2(n-1)) q_2(n) ) + \frac{1}{4}( p_1(n-1) q_3(n) + p_2(n-1) q_1(n) )$
\item[(vii)] $q_3(n) = \frac{3}{4}( p_1(n-1) q_3(n) + p_2(n-1) q_1(n) ) + \frac{1}{4} ( (p_1(n-1) + p_2(n-1)) q_2(n) + p_2(n-1) )$
\end{itemize}
\end{Lemma}

\begin{proof} 
Consider Puzzle~$1\to a$. 
Note that $D_n$ cannot move unless $D_{n-1}^1$ are on the same peg.
Therefore, we can focus only on $D_{n-1}^1$ until they all come to the same peg, 
which will happen (on average) after $E_{1\to a}(n-1)$ moves. 
This will be the first peg (where $D_n$ is) with probability $p_1(n-1)$,  
in which case we have the final state with all disks on the first peg.
Otherwise, with the probability $1-p_1(n-1)=2 p_2(n-1)$, 
we have $D_n$ on the first peg and $D_{n-1}^1$ on 
a different peg (equally likely on the second or the third one), in which case the remaining moves can be considered as 
as solving an instance of Puzzle~$\nicefrac12\to a$.
This proves formula (i).

From the above it is also easy to see that in the final state all disks 
will be at the first peg with the probability
$p_1(n-1) + 2 p_2(n-1) q_2(n)$ and at the second peg or third peg
with the same probability $p_2(n-1) q_1(n) + p_2(n-1) q_3(n)$, which proves formulae (ii) and (iii).

Now, consider Puzzle~$\nicefrac12\to a$.
Moves in this puzzle can be split into two or three stages as follows.
In Stage 1 only $D_n$ is moving (between the second and third pegs),
Stage 2 starts with a move of $D_1$ (from the top of the first peg)
and ends when $D_{n-1}^1$ are on the same peg. 
If this is not the final state, the remaining moves are viewed as Stage 3. 
Let us analyze these stages.

It is easy to see that the expected number of moves in Stage 1 is
\(
\tfrac{2}{3} \cdot \left( \left(\tfrac{1}{3}\right)^0 \cdot 0 + 
\left(\tfrac{1}{3}\right)^1 \cdot 1 + \dots \right) = \tfrac{1}{2}
\)
and with the probability 
\(
\tfrac{2}{3} \cdot \left( \left(\tfrac{1}{3}\right)^0 + 
\left(\tfrac{1}{3}\right)^2 + \dots \right) = \tfrac{3}{4}
\)
it will end up at the same peg where it started, namely the second peg.
The probability for $D_n$ to end up at the third peg is therefore $1 - \tfrac{3}{4} = \tfrac{1}{4}$.
The expected number of moves in Stage 2 is simply $E_{1\to a}(n-1)$ and 
at the end $D_{n-1}^1$ are on the first peg with the probability $p_1(n-1)$ and on 
the second or third pegs with the equal probability $p_2(n-1)$.
Therefore, with the probability $p_2(n-1)$ we are at the final state
(no matter where $D_n$ is left after Stage 1) and with the probability
$1-p_2(n-1)=p_1(n-1)+p_2(n-1)$ we embark upon Stage 3, which can be viewed simply as a new 
instance of Puzzle~$\nicefrac12\to a$ with the expected number 
of moves $E_{\nicefrac12\to a}(n)$. The above analysis proves formulae (iv)-(vii).
\end{proof}

Formula~\eqref{FN3} for $E_{1\to a}(n)$ follows directly from (i) and (iv).
Namely, formula (iv) can be rewritten as $p_2(n-1) E_{\nicefrac12\to a}(n) = \tfrac{1}{2} + E_{1\to a}(n-1)$.
Substituting this into (i) results in the recurrent formula:
$$E_{1\to a}(n) = E_{1\to a}(n-1) + 2 p_2(n-1) E_{\nicefrac12\to a}(n) = 3 E_{1\to a}(n-1) + 1.$$
Together with $E_{1\to a}(1)=1$ this formula proves \eqref{FN3}, which in turn further implies
\begin{equation}\label{F12a}
E_{\nicefrac12\to a}(n) = \frac{\tfrac{1}{2} + E_{1\to a}(n-1)}{p_2(n-1)} = \frac{3^{n-1}}{2 p_2(n-1)}.
\end{equation}

Let us focus on the recurrent 
equations (ii), (iii), (v), (vi), (vii) and solve them
with respect to $p_1(n)$, $p_2(n)$, $q_1(n)$, $q_2(n)$, and $q_3(n)$.
Solving Puzzle~$r\to a$ and Puzzle~$1\to 3$ for $n=2$, we easily obtain the following initial conditions:
$$
p_1(2) = \nicefrac{5}{8},\quad p_2(2) = \nicefrac{3}{16},\quad 
q_1(2) = \nicefrac{1}{8},\quad q_2(2) = \nicefrac{5}{8},\quad q_3(2) = \nicefrac{1}{4}.
$$
Also solving Puzzle~$r\to a$ for $n=1$, we get $p_1(1) = 0$ and $p_2(1)=\nicefrac{1}{2}$.

From (v) we have
$p_2(n-1) q_3(n) = q_1(n) - p_1(n-1) q_1(n) = (1-p_1(n-1)) q_1(n) = 2 p_2(n-1) q_1(n)$.
Since $p_2(n-1)>0$ for all $n\geq 2$, we also have 
$q_3(n) = 2 q_1(n)$, $q_2(n) = 1 - q_1(n) - q_3(n) = 1 - 3 q_1(n)$,
and $(8 - 3 p_1(n-1)) q_1(n) = 1$ for all $n\geq 2$.

Using these relations, we simplify equation (ii) to
$$
\begin{array}{l}
p_1(n) = p_1(n-1) + 2 p_2(n-1) (1 - 3 q_1(n)) = 1 - 6 p_2(n-1) q_1(n)\\
= 1 - 3(1-p_1(n-1)) q_1(n) = 1 - 3 q_1(n) + 3 p_1(n-1) q_1(n)\\
= 1 - 3 q_1(n) + 8 q_1(n) - 1 = 5 q_1(n).
\end{array}
$$

Combining the above equations, we have $(8 - 15 q_1(n-1)) q_1(n) = 1$, that is
\begin{equation}\label{q1recurrent}
q_1(n) = \frac{1}{8 - 15q_1(n-1)}.
\end{equation}

\begin{Lemma} For all positive integers $n$,
\begin{equation}\label{q1explicit}
q_1(n) = \frac{5^{n-1} - 3^{n-1}}{5^n - 3^n}.
\end{equation}
\end{Lemma}

\begin{proof} We prove formula for $q_1(n)$ by induction on $n$. 

For $n=1$, formula \eqref{q1explicit} trivially holds as $q_1(1) = 0$.
Now for integer $m\geq 1$, if formula \eqref{q1explicit} holds for $n=m$, then using \eqref{q1recurrent} we get
$$
q_1(m+1) = \frac{1}{8 - 15q_1(m)} = \frac{1}{8-15\frac{5^{m-1} - 3^{m-1}}{5^m - 3^m}} = \frac{5^m-3^m}{8(5^m-3^m) - 15(5^{m-1} - 3^{m-1})} 
= \frac{5^m-3^m}{5^{m+1} - 3^{m+1}}.
$$
Therefore, formula \eqref{q1explicit} holds for $n=m+1$, which completes the proof.
\end{proof}

Formula \eqref{q1explicit} further implies:
$$q_2(n) = 1 - 3 q_1(n) = \frac{2\cdot 5^{n-1}}{5^n - 3^n};$$
$$q_3(n) = 2 q_1(n) = \frac{2\cdot (5^{n-1} - 3^{n-1})}{5^n - 3^n};$$
$$p_1(n) = 5 q_1(n) = \frac{5^n - 5\cdot 3^{n-1}}{5^n - 3^n};$$
$$p_2(n) = \frac{1 - p_1(n)}{2} =  \frac{3^{n-1}}{5^n - 3^n}.$$
The last formula together with \eqref{F12a} proves formula \eqref{FN4}.

Now we are ready to prove formulae \eqref{FN1} and \eqref{FN2}.
Lemma~\ref{ThP1} together with $E_{r\to a}(0)=0$ implies
$$E_{r\to a}(n) = \sum_{k=1}^n \left( E_{r\to a}(k)-E_{r\to a}(k-1) \right) = \sum_{k=1}^n \left( 5^{k-1} - 3^{k-1} \right) = \frac{5^n - 2\cdot 3^n + 1}{4}.$$
Lemma~\ref{ThP2} implies
$$E_{1\to 3}(n) = \frac{E_{1\to a}(n)}{p_2(n)} = \frac{(3^n - 1)(5^n - 3^n)}{2\cdot 3^{n-1}}.$$

Finally, we derive formula \eqref{FN5}.
Solving Puzzle~$r\to 1$ can be viewed as first solving Puzzle~$r\to a$ and if it does not result in 
all disks on the first peg (that happens with the probability $\nicefrac{2}{3}$), continue solving it as Puzzle~$1\to 3$. 
Therefore, the expected number of moves in Puzzle~$r\to 1$ is:
$$E_{r\to 1}(n) = E_{r\to a}(n) + \frac{2}{3} E_{1\to 3}(n) = \frac{5^n - 2\cdot 3^n + 1}{4} + \frac{(3^n - 1)(5^n - 3^n)}{3^n} 
= \frac{5^{n+1}-2\cdot 3^{n+1} + 5}{4} - \left(\frac{5}{3}\right)^n.$$

\section{Analysis of Puzzle~$1\to 3$ via Networks of Electrical Resistors}

We now present an altogether different method for solving Puzzle~$1\to 3$ that relies on 
on its interpretation as a random walk between two corner nodes in the corresponding Sierpinski gasket and a result from
electrical circuit theory.
The corner nodes of the Sierpinski gasket for the Tower of Hanoi with $n$ disks correspond to 
the states with all disks on the first, second, and third peg, which we label $A(n)$, $B(n)$, $C(n)$, respectively.
In other words,
$A(n)=(D_n^1, \emptyset, \emptyset)$, $B(n)=(\emptyset,D_n^1,\emptyset)$, and $C(n)=(\emptyset,\emptyset,D_n^1)$, where $\emptyset$ is the empty set. 

A \emph{random walk} on an undirected graph consists of a sequence of steps from one end of an edge to the other end of that edge. 
If the random walker currently is in a
state $S$ that has a total of $M$ distinct states that can be reached from $S$ in one step, then the random walker's next step 
will go from $S$ to each of these $M$ states with probability $\nicefrac{1}{M}$. 
Among the $3^n$ states of the Tower of Hanoi with $n$ disks
$3^n - 3$ have $M=3$; the other 3, namely the corner nodes $A(n)$, $B(n)$, $C(n)$, have $M=2$.

Building on a monograph by P. G. Doyle and J. L. Snell~\cite{DoyleSnell84}, A. K. Chandra et al.~\cite{Chandraetal89} proved the 
following theorem:

\begin{Th}[The Mean Commute Theorem]\label{Commute}
The expected number of steps in a cyclic random walk on an undirected graph that starts from any vertex $V$,
visits vertex $W$, and then returns to $V$ equals $2mR_{\rm VW}$, where $m$ is the number of edges in the graph and 
$R_{\rm VW}$ is the electrical resistance between nodes $V$ and $W$ when a 1-ohm resistor is inserted in every
edge of the graph.
\end{Th}

Fig.~\ref{Fig1}a shows the graph for the 1-disk Tower of Hanoi with a 1-ohm resistor inserted in each of its three edges. There are two parallel paths between
states $A$ and $C$, which are respectively the initial and final states for the Puzzle~$1\to 3$ with $n=1$.
The direct path along edge $\overline {\rm AC}$ has a resistance of 1 ohm, and the indirect path along edge $\overline{\rm AB}$ 
followed by edge $\overline{\rm BC}$ has a resistance of 2 ohms.  
The overall resistance from $A$ to $C$ therefore is $\nicefrac{1\cdot 2}{(1 + 2)} = \nicefrac{2}{3}$ ohm. 
Since there are 3 edges in the graph, the mean commute time from $A$ to $C$ and back is $2 \cdot 3 \cdot \nicefrac{2}{3} = 4$.  
By symmetry,\footnote{Full symmetry is required to justify equal mean lengths of the outbound and return segments of a commute. 
E.g., a 3-vertex graph with only two edges, $\overline{FG}$ and $\overline{GH}$ has full symmetry from $F$ to $H$ and back to $F$, 
but not from $F$ to $G$ and back to $F$. Simple calculations give $E_{FH}=E_{HF}=4$, but $E_{FG}=1$ whereas $E_{GF}=3$.}
on average half of the time is spent going from $A$ to $C$ and the other half returning from $C$ to $A$.  
Accordingly, the mean time it takes a randomly moving Tower of Hanoi with $n=1$ to reach peg 3 starting from 
the first peg is $\nicefrac{4}{2} = 2$, which agrees with formula \eqref{FN2} for $E_{1\to 3}(n)$ when $n=1$.

\begin{figure}
\centering \begin{tabular}{lcl}
a) & ~\qquad\qquad~ & b) \\
\includegraphics[height=.3\textwidth]{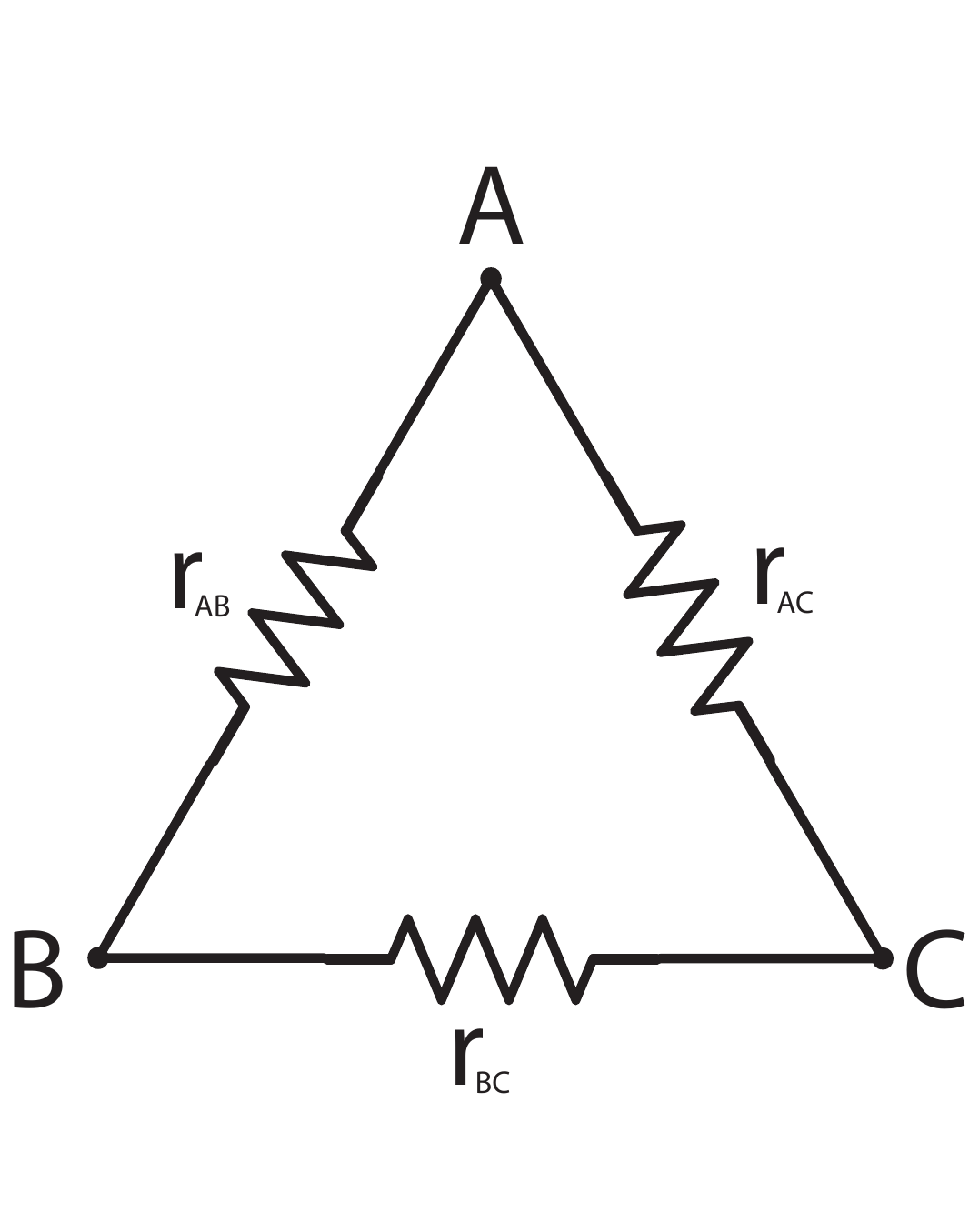} 
& ~\qquad\qquad~ &
\includegraphics[height=.3\textwidth]{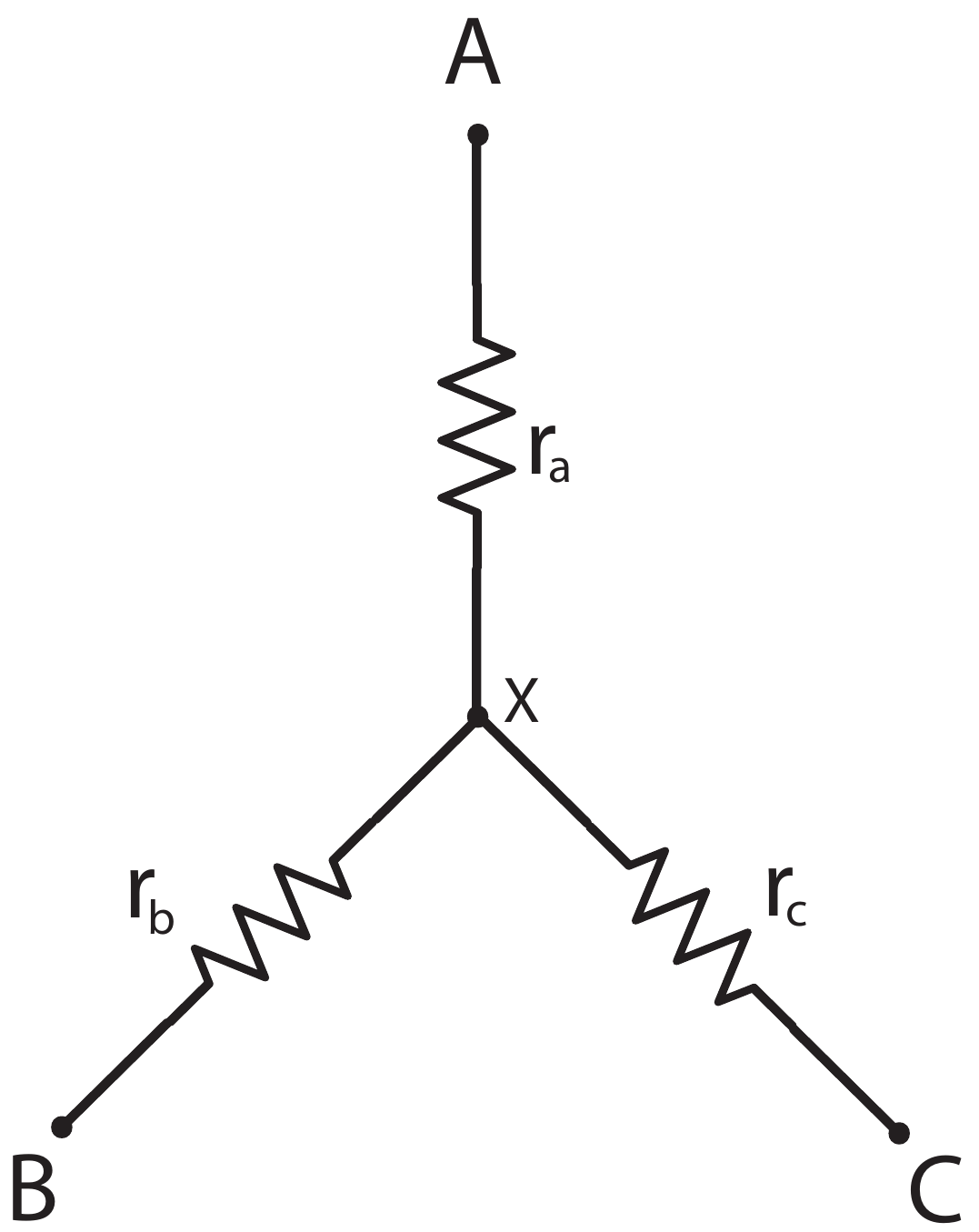}
\end{tabular}
\caption{\textbf{a)} 3-Resistor Delta; \textbf{b)} 3-Resistor Wye.}
\label{Fig1}
\end{figure}

We proceed to iterate this approach in order to derive formula \eqref{FN2} for general $n$.  The key to performing the requisite iterations is
the classical delta-to-wye transformation of electrical network theory \cite{Kennelly1899}.  
A ``delta" is a triangle with with vertices $A$, $B$, and $C$ that 
has resistances $r_{\rm AB}$ in edge $\overline {\rm AB}$, 
$r_{\rm AC}$ in edge $\overline {\rm AC}$, and $r_{\rm BC}$ in edge $\overline {\rm BC}$ (Fig.~\ref{Fig1}a).
The corresponding ``wye" (Fig.~\ref{Fig1}b) has the same three
nodes $A$, $B$, and $C$ plus a fourth node $x$ and three edges $\overline {\rm Ax},
\overline {\rm Bx}$, and $\overline {\rm Cx}$ that contain respective resistances $r_a, r_b$ and $r_c$.  
It is straightforward to verify that, if
\begin{equation} 
r_a = \tfrac {r_{\rm AB}r_{\rm AC}} {r_{\rm AB} + r_{\rm AC} + r_{\rm BC}},\quad 
r_b = \tfrac {r_{\rm AB}r_{\rm BC}} {r_{\rm AB} + r_{\rm AC} + r_{\rm BC}},\quad{\rm and}\quad
r_c = \tfrac {r_{\rm AC}r_{\rm BC}} {r_{\rm AB} + r_{\rm AC} + r_{\rm BC}}, 
\end{equation}
then the net resistance $R_{\rm AB}$ between nodes $A$ and $B$ will be the same in Fig.~\ref{Fig1}b as it is in Fig.~\ref{Fig1}a, and likewise for the 
net resistances $R_{\rm AC}$ between nodes $A$ and $C$ and $R_{\rm BC}$ between nodes $B$ and $C$.

We shall need to consider only the special case $r_{\rm AB} = r_{\rm AC} = r_{\rm BC} = R$, 
in which $r_a =r_b =r_c = \nicefrac{R}{3}$.  In particular, when $R=1$, we have $r_a = r_c = \nicefrac{1}{3}$, 
so Fig.~\ref{Fig1}b yields $R_{\rm AC} = \nicefrac{2}{3}$,
the same result we obtained before by considering the two parallel paths from $A$ to $C$ in Fig.~\ref{Fig1}a.

\begin{Th}[Delta-to-Wye Induction] \label{Wye Induction} 
The state diagram for the $n$-disk Tower of Hanoi
with a unit resistance in each of its branches can be converted, for purposes of determining the resistance between any two 
of its three corner nodes $A(n)$, $B(n)$ and $C(n)$, into a
simple ``wye" in which $A(n)$, $B(n)$ and $C(n)$ each are connected to a center point by links that each contain a common amount 
of resistance denoted by $R(n)$.
\end{Th}

\begin{proof}
We prove Theorem~\ref{Wye Induction} by induction on $n$.  We have already shown that it is true for $n=1$,
the value of $R(1)$ being $\nicefrac{1}{3}$ ohm.  We now show that if Theorem~\ref{Wye Induction}'s statement is true for some positive
integer $n$, then it must also be true for $n+1$. 

\begin{figure}[!t]
\begin{tabular}{lcl}
a) & ~ & b)\\
\includegraphics[height=.28\textheight]{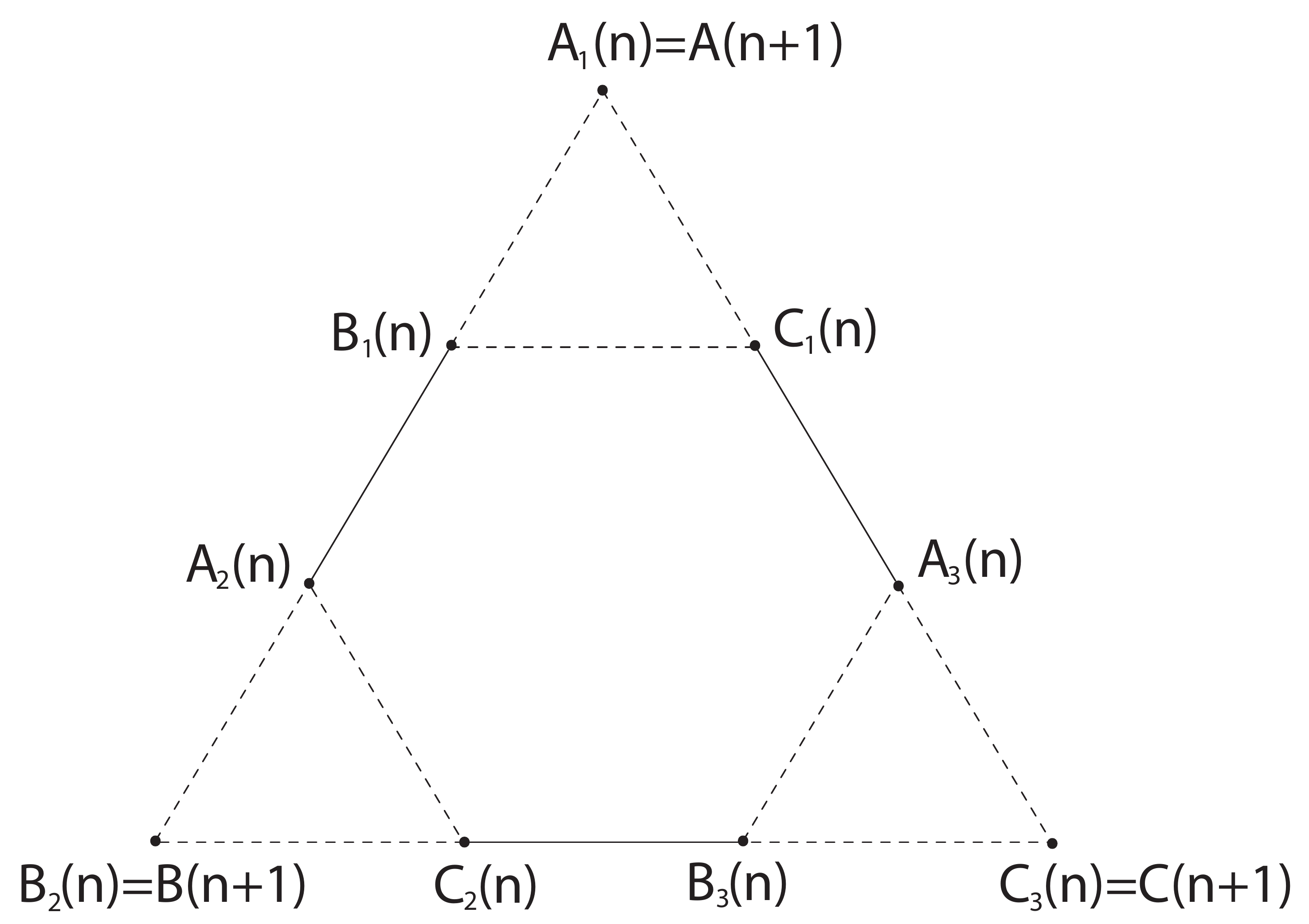} 
& ~ &
\includegraphics[height=.28\textheight]{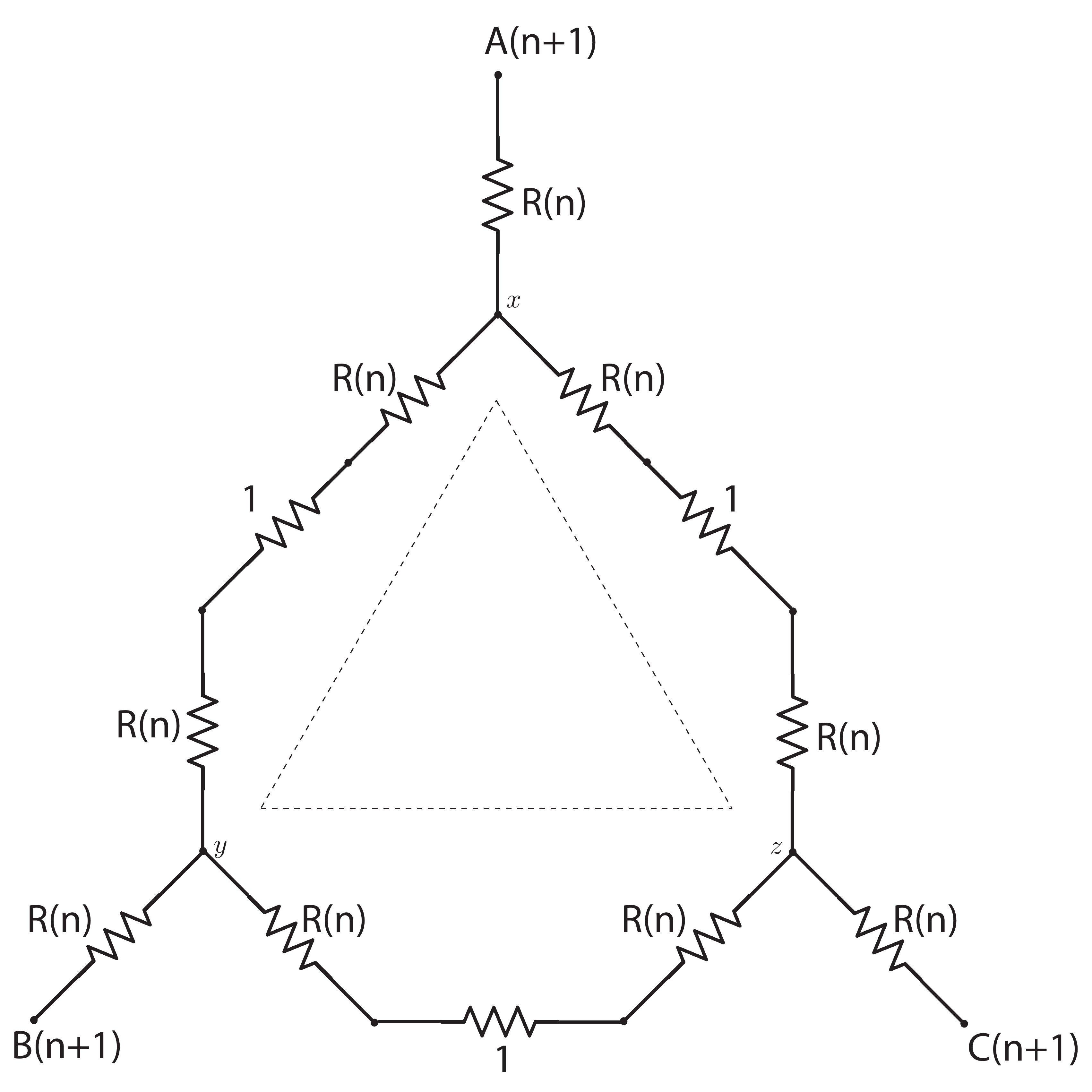} 
\end{tabular}
\caption{\textbf{a)} State transition diagram and \textbf{b)} the corresponding network of resistors for the Tower of Hanoi with $n+1$ disks.}
\label{Fig34}
\end{figure}

As noted earlier, the state transition diagram of the Tower of Hanoi with $n+1$ disks is produced by generating three replicas of that for the Tower of Hanoi
with $n$ disks that possess respective corner nodes $\{A_1(n),B_1(n),C_1(n)\}$, $\{A_2(n),B_2(n),C_2(n)\}$ and $\{A_3(n),B_3(n),C_3(n)\}$ and then
adding three bridging links: one between $B_1(n)$ and $A_2(n)$, another between $C_1(n)$ and $A_3(n)$, 
and the third between $C_2(n)$ and $B_3(n)$ (Fig.~\ref{Fig34}a). The corner nodes in the resulting graph are $A_1(n)=A(n+1)$, $B_2(n)=B(n+1)$ and 
$C_3(n)=C(n+1)$, the only three nodes in the replicas to which none of the bridging links is incident.
When applying Theorem~\ref{Commute}, a unit resistance also must be inserted in each of the three bridging links,
just as is the case for every other link in the graph.  By the induction hypothesis,
the resistance between any two nodes in $\{A_1(n),B_1(n),C_1(n\}$ can be computed using a wye comprised of links from each of them to a center point, call it $x$, 
each of these links having resistance $R(n)$.
The same is true for any two nodes in $\{A_2(n),B_2(n),C_2(n)\}$ and any two in $\{A_3(n), B_3(n), C_3(n)\}$, with the respective center points called $y$ and $z$.
Doing these three delta-to-wye conversions results in Fig.~\ref{Fig34}b.
Note that triangle $\overline{xyz}$ in this figure is a delta, each edge of which has resistance $R(n) + 1 + R(n)$ = $2R(n) +1$.  
The delta-to-wye transformation applied to this delta network results in the wye network of Fig.~\ref{Fig5}.  Each of the links in this wye has
the same resistance, namely
\begin{equation} \label{R(n+1)}
R(n+1) = R(n) + \frac{2R(n) + 1}3  =  \frac{5R(n)+1}{3}.
\end{equation}
Theorem~\ref{Wye Induction} is proved.
\end{proof}

\begin{figure}[!t]
\centering \begin{tabular}{lcl}
a) & ~\qquad~ & b)\\
\includegraphics[height=.3\textheight]{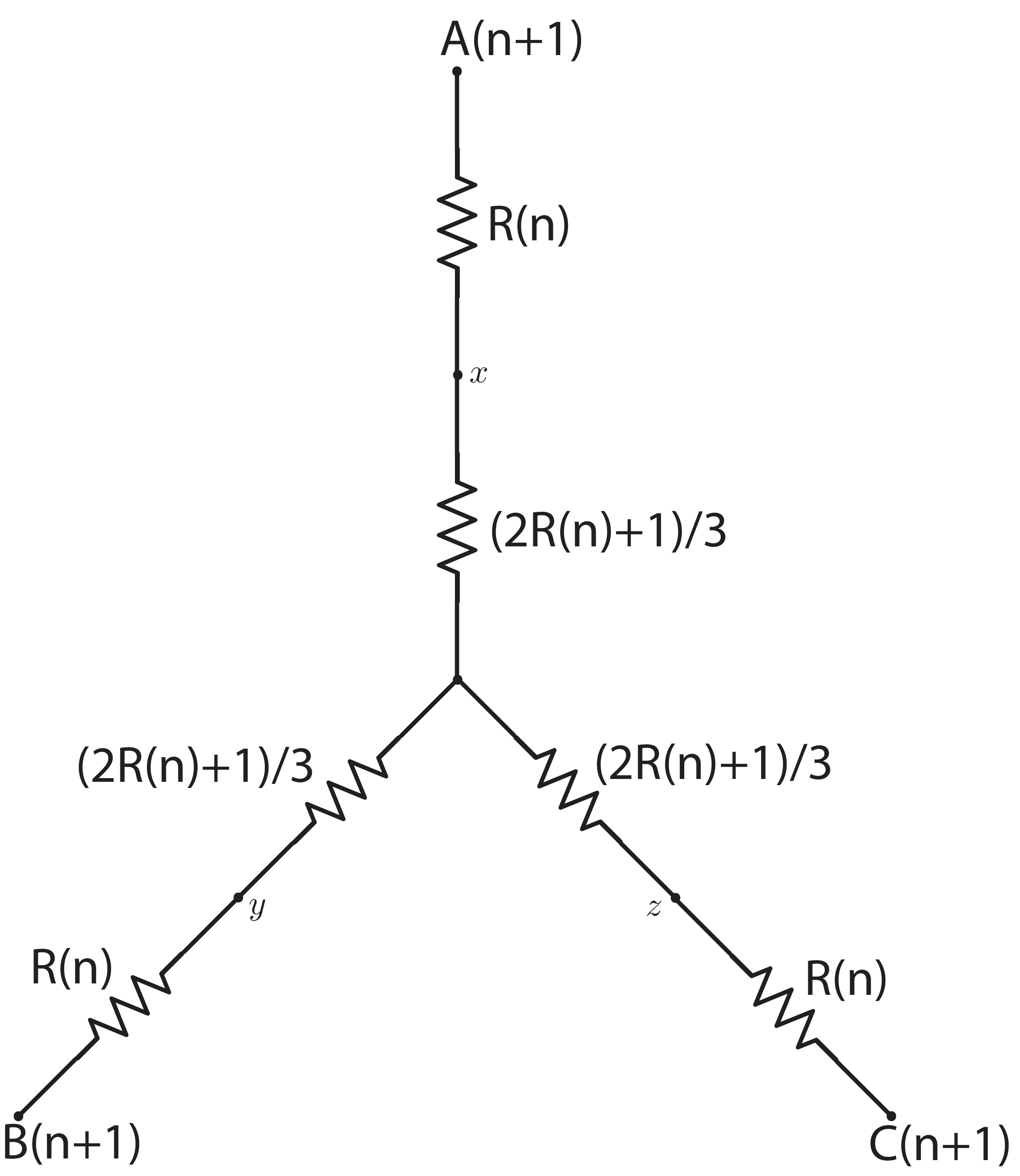}
& ~\qquad~ &
\includegraphics[height=.3\textheight]{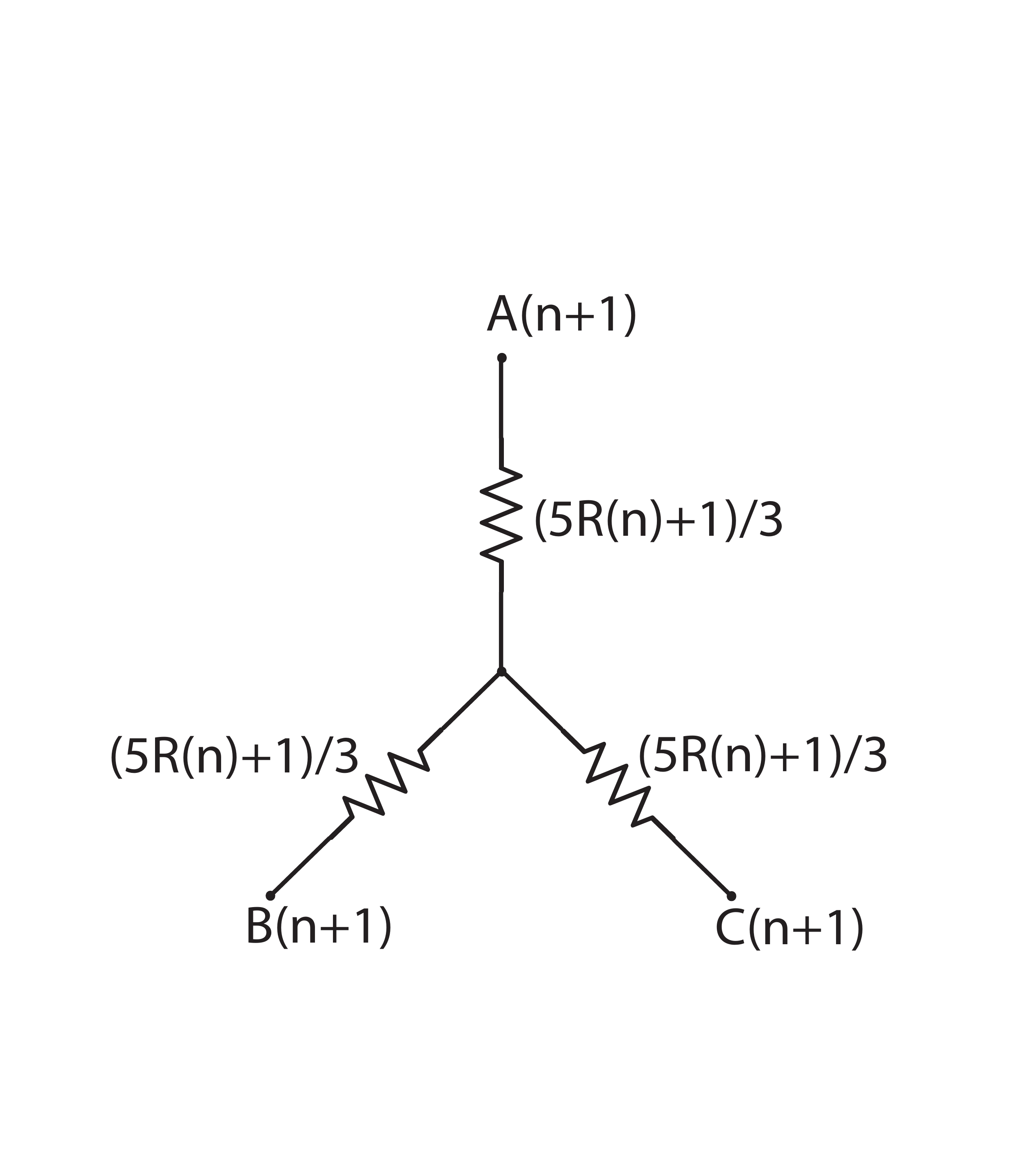}
\end{tabular}
\caption{\textbf{a)} A single wye of the state transition diagram of the Tower of Hanoi with $n+1$ disks and \textbf{b)} its reduction
to a simpler wye with resistance $\nicefrac{(5R(n)+1)}{3}$ in each link.}
\label{Fig5}
\end{figure}

From equation \eqref{R(n+1)} and the boundary condition $R(1) = \nicefrac{1}{3}$, we obtain the key result:
\begin{equation} \label{R_n}
R(n) = \frac{5^n-3^n}{2 \cdot 3^n}.
\end{equation}
The number $m_n$ of edges in the state transition diagram for the $n$-disk Tower of Hanoi is
\[ 
m_n = \frac{(3^n -3)\cdot 3 + 3\cdot 2}{2} = \frac{3}{2}(3^n -1).
\]
From the Mean Commute Time theorem and the symmetry of random walks from $A$ to $C$
and from $C$ to $A$, it follows that the mean number of steps it takes a randomly 
moving $n$-disk Tower of Hanoi to transfer all its disks from the first peg to the third peg equals
\begin{equation} \label{AlekseyevFormula}
m_n \cdot R_{AC}(n) = m_n \cdot 2R(n) = \frac{(3^n - 1)(5^n - 3^n)}{2 \cdot 3^{n-1}},
\end{equation}
which agrees with formula \eqref{FN2} for the mean number $E_{1\to 3}(n)$ of moves in Puzzle~$1\to 3$.

\section{Discussion}

The minimum number of moves required to solve the Tower of Hanoi with $n=64$ disks is ``only" $2^{64} - 1 = 18,446,744,073,709,551,615$.
Since it is often asserted that monks possess superhuman abilities, maybe they can move disks rapidly.  
Perhaps they can make a move a microsecond, maybe even a move a nanosecond, and planet Earth may expire any day now. 
This in part motivated adoption of a randomly moving Tower of Hanoi~\cite{Berger08}.

Formula~\eqref{AlekseyevFormula} shows that replacing the minimum-moves strategy with a random walk forestalls the end of the world by a factor of roughly 
$\left(\tfrac{5}{2}\right)^{64} > 2.9 \times 10^{25}$ on average.  Although this is reassuring, it nonetheless would be further comforting to know that 
the coefficient of variation of the random number of steps in Puzzle~$1\to 3$ with $n=64$ disks is small,
i.e., that its standard deviation is many times smaller than its mean.  
Exact determination of said coefficient of variation is an open problem that we may address in future research.  

An extensive bibliography of some 370 mathematical articles concerning the Tower of Hanoi puzzle and variations 
thereon has been complied by Paul Stockmeyer~\cite{Stockmeyer05}. 
While the current paper was under review, our attention was drawn to the work \cite{Wu2011}, which develops similar ideas of analyzing 
random walks in Sierpinski gaskets via resistor networks. 
Properties the generalized Tower of Hanoi with more than three pegs and its state transition diagrams are studied to some extent in~\cite{Hinz2013}.

\section*{Acknowledgments}
The authors are indebted to Neil J. A. Sloane, creator and caretaker of the Online Encyclopedia of Integer Sequences (OEIS)~\cite{OEIS}.
Neil's interest in the research reported herein and the existence of his OEIS connected us to one another and to the broader Tower of Hanoi research community.
The authors are also thankful to Sergey Aganezov and Jie Xing for their help with preparation of the figures.

The first author was supported by the National Science Foundation under grant No. IIS-1253614.

\bibliographystyle{plain}
\bibliography{hanoi.bib}

\end{document}